\documentclass{article}
\usepackage{amsmath, tikz-cd}
\usepackage{multicol}
\usepackage{subfigure}
\usepackage{amssymb}
\usepackage{enumitem}
\usepackage{amsthm}

 \newtheorem{thm}{Theorem}
 
 \newtheorem{quest}[thm]{Question}

 \newtheorem{exm}[thm]{Example}
 \newtheorem{dfn}[thm]{Definition}

\begin{document}
\title{One Parameter Semigroups in Two Complex Variables.}
 
\author{Michael R. Pilla\\ Indiana University, Bloomington\\mpilla@iu.edu}

\maketitle 

\begin{abstract}
For self maps of the disk, it can be shown that under the right conditions one can embed a discrete iteration of the map into a continuous semigroup. In this article we extend these results to two complex variables for maps of the unit ball into itself under some restricted conditions.
\end{abstract}

\section{Introduction}

One parameter semigroups of analytic self maps of the disk are a topic of interest in complex analysis in that, when they exist, they allow the embedding of a discrete iteration of the analytic map into a time-continuous analogue. In one complex variable, of course, the Riemann Mapping Theorem tells us that the study of such semigroups on the unit disk is sufficient. In this case, such objects can be ``linearized" analytically and viewed as continuous semigroups of linear fractional maps \cite{abate}, \cite{berkson}. One may generalize this study of semigroups to the unit ball $B_N=\{(z_1,...,z_N) \in \mathbb{C}^N \mid \sum_{i=1}^N|z_i|^2<1\}$ in $\mathbb{C}^N$. In several variables, the problem is necessarily more complicated, although some properties are known \cite{defab}.  In this article, we will focus on explicit constructions of one parameter semigroups for analytic self maps of the unit ball $\mathbb{B}_2$ using a model theory of linear fractional maps described by Cowen in \cite{cowen}.

Consider the set 

$$T=\{\phi : B_N \rightarrow B_N \mid \phi \hspace{1.5mm} \text{is a nonconstant analytic map, not an automorphism}\}.$$ 

For $\phi \in T$, let $\phi_n$ represent the $n^{\text{th}}$ iteration of $\phi$. It is clear that the set of iterates $\{\phi_n\}$ under composition for $n=0,1,2,...,$ defines a discrete semigroup. 

Recall that a \textbf{one parameter semigroup} for a monoid $(S, *)$ is a map $\phi: [0, \infty) \rightarrow S$, such that 

 \begin{enumerate}[label=\roman*.]
   \item $\phi(0)=I$.
   \item $\phi(s+t)=\phi(s) \ast \phi(t)$.
 \end{enumerate}

For a full treatment of one parameter semigroups, see, for example, \cite{engel}. 

The strategy will be to use a model theory of linear fractional maps to transport the iterates of our analytic maps to a set of model linear fractional maps. It can be shown that a nonconstant analytic self map $\phi$ of the disk, not an automorphism, can be intertwined with a linear fractional map $\Phi$ and an analytic map $\sigma$ such that

\begin{equation*}
\Phi \circ \sigma=\sigma \circ \phi
\end{equation*}

\noindent where $\sigma$ maps the disk into a domain $\Omega$, which we call the characteristic domain, with $\Phi$ mapping $\Omega$ onto $\Omega$ \cite{cowen}. We have the following commutative diagram:

$$\begin{tikzcd}
D \arrow[r, "\phi"] \arrow[d, "\sigma"]
& D \arrow[d, "\sigma"] \\
\Omega \arrow[r, "\Phi" ]
& \Omega
\end{tikzcd}$$

If $\Omega$ is the smallest set containing $\sigma(D)$ for which $\Phi(\Omega)=\Omega$, then the model parameters $(\sigma, \Omega, \Phi)$ will be unique up to holomorphic equivalence. The classification depends heavily on the behavior near the Denjoy-Wolff point $a$. A key fact we note is that $\sigma \circ \phi = \Phi \circ \sigma$ implies $\sigma \circ \phi_n = \Phi_n \circ \sigma$ for $\Phi$ a linear fractional map. Under the right conditions imposed on $\phi$, one can use this model theory to extend the discrete semigroup $\{\phi_n\}$ to a one parameter semigroup in one complex variable \cite{cowen}. We aim to generalize this result to two complex variables.

\subsection{Linear Fractional Maps in $\mathbb{C}^N$}

In order to obtain one parameter semigroups of analytic maps in several variables using linear fractional model maps, we must generalize the class of linear fractional maps to higher dimensions. We will take the perspective that the associated matrices of our linear fractional maps should act as linear transformations on complex projective coordinates \cite{pilla}. Using this perspective brings us to the following definition.

\begin{dfn}
We say $\phi$ is a linear fractional map in $\mathbb{C}^N$ if 

\begin{equation*}\label{LFM}
\phi(z)=\frac{Az+B}{\langle z,C \rangle+D}
\end{equation*}

\noindent where $A$ is an $N \times N$ matrix, $B$ and $C$ are column vectors in $\mathbb{C}^N$, $D \in \mathbb{C}$, and $\langle \cdot, \cdot \rangle$ is the standard inner product.
\end{dfn}

With this definition, we recover a variety of properties that the class of linear fractional maps enjoy in one variable  \cite{lfm}. 

We also define the associated matrix to our linear fractional maps:

\begin{dfn}
The associated matrix $m_{\phi}$ of the linear fractional map $\phi(z)=\frac{Az+B}{\langle z,C \rangle+D}$ is given by

\begin{equation*}
    m_{\phi} = \begin{pmatrix}
        A & B \\
        C^* & D
     \end{pmatrix}.
\end{equation*}
\end{dfn}

It's not hard to show that $m_{\phi_1 \circ \phi_2}=m_{\phi_1}m_{\phi_2}$ and $m_{\phi^{-1}}=(m_{\phi})^{-1}$ from which we see $m_{\phi_n}=(m_{\phi})^n$. 

The utility of the associated matrix is that it allows us to convert iteration into multiplication, allowing a clearer view of how to embed our discrete semigroup $\{\phi_n\}$ into a continuous one.

\subsection{A Model Theory in Two Complex Variables}

In the case where $\phi$ is an analytic self map of the unit ball $\mathbb{B}_N$ with no interior fixed points, MacCluer demonstrated the existence of a unique fixed point on the boundary such that the iterates of $\phi$ converge uniformly to $a$ on compact subsets of $\mathbb{B}_N$ \cite{maccluer}. As in the disk, we call this priviliged point the Denjoy-Wolff point. Likewise, our classification in $\mathbb{B}_2$ will depend on this point.

Due to its critical use of the Riemann Mapping Theorem, the model theory in the disk cannot be generalized to $\mathbb{B}_N$ in it's full generality. One can show, however, that a classification can be obtained for the class of linear fractional maps in two complex variables \cite{crosby}. As in the disk, one can show that this classification is invariant under conjugation by an automorphism. The seven cases obtained are determined by the behavior of the map $\phi$ near the Denjoy-Wolff point and its characteristic domain. The cases depend on whether the Denjoy-Wolff point is an interior fixed point or whether it is on the boundary. They also depend on the multiplicity of the Denjoy-Wolff point and its associated characteristic domain. We find that there are three characteristic domains to be considered. These are the whole space $\mathbb{C}^2$, the half space $\mathbb{H}=\{(z_1, z_2) \in \mathbb{C}^2 \mid \Re z_1 > 0\}$, and the Siegel half space $\mathbb{H}^2=\{(z_1, z_2) \in \mathbb{C}^2 \mid \Re z_1 > |z_2|^2 \}$. We reproduce the results of \cite{crosby} here for convenience.

\begin{thm}[The Model for Iteration of Linear Fractional Maps]\label{data}
Let $\phi$ be a linear fractional map of $\mathbb{B}_2$ into itself, not an automorphism of the ball and not constant. We can intertwine $\phi$ with a model linear fractional map $\Phi$ with characteristic domain $\Omega$, either the half space, Siegel half space, or the whole space, and an open map $\sigma$ from $\mathbb{B}_2$ into $\Omega$ such that 

\begin{equation*}
\sigma \circ \phi=\Phi \circ \sigma.
\end{equation*}

If $\Omega$ is the smallest set containing $\sigma(\mathbb{B}_2)$ for which $\Phi(\Omega)=\Omega$, then the model parameters $(\sigma, \Omega, \Phi)$ will be unique up to holomorphic equivalence. 
\end{thm}

In addition, there exists a set $V$, known as the fundamental set, such that $V$ is an open, connected, simply connected subset of $\mathbb{B}_2$ such that $\phi(V) \subset V$ and for every compact set $K$ in $\mathbb{B}_2$, there is a positive integer $n$ with  $\phi_n (V) \subset V$ with $\phi$ and $\sigma$ univalent on $V$ and with $\sigma(V)$ a fundamental set for $\Phi$ on $\Omega$.

We also consider maps that are analytic, not necessarily linear fractional, that also have this model for iteration. Since all analytic self maps of the disk, not automorphisms, have such a model, one's intuition is that analytic maps with a model for iteration share more common behavior with analytic self maps of the disk than generic self maps of the unit ball. 

\section{Constructing One Parameter Semigroups in Two Complex Variables}

We began by considering one parameter semigroups for the class of linear fractional maps from the unit ball into itself in two complex variables. Richman demonstrated a criteria to determine when a linear fractional map maps $\mathbb{B}_N$ into itself \cite{rich}. In addition to the model theory, one can classify our class of linear fractional maps according to what is known as the boundary dilation coefficient \cite{bisi}. Semigroups have been classified, up to conjugation, in several complex variables using this coefficient  \cite{bracci}. In this section, we explicitly construct the semigroups based on the classification using the model theory of linear fractional maps.  Recall that $\sigma \circ \phi = \Phi \circ \sigma$ implies $\sigma \circ \phi_n = \Phi_n \circ \sigma$. In the case where $\phi$ is a linear fractional map, it is sufficient to take $\sigma$ to be an invertible linear fractional map and we may write

\begin{equation}
\phi=\sigma^{-1} \circ \Phi \circ \sigma.
\end{equation}

We can thus define a discrete semigroup for the set $\{\phi_n\}$ by

\begin{equation}
\phi_n=\sigma^{-1} \circ \Phi_n \circ \sigma
\end{equation}

\noindent where $\Phi_n$ is our model linear fractional map. It's not hard to see that the eigenvectors of the associated matrix $m_{\phi}$ correspond to fixed points of $\phi$.  We assume our maps are invertible and thus we do not have to consider zero eigenvalues. We will make heavy use of Jordan form to attain our goals. We factor the associated matrix $m_{\phi}$ to obtain $m_{\phi}=S\Lambda S^{-1}$ where the columns of $S$ are (generalized) eigenvectors of $m_{\phi}$ and $\Lambda$ is in Jordan form. Given a linear fractional map  $\phi$ of $\mathbb{B}_N$ into itself and an automorphism $\psi$ of $\mathbb{B}_N$, we see that

$$m_{\psi}m_{\phi}m_{\psi}^{-1}=m_{\psi}S \Lambda S^{-1}m_{\psi}^{-1}=(m_{\psi}S) \Lambda (m_{\psi}S)^{-1}.$$

Thus, not only are our maps are equivalent up to conjugation by an automorphism, but conjugation by an automorphism yields the same Jordan form matrix $\Lambda$. 

Letting $\phi_0$ be the identity, we see that for $n=0,1,2, \dots$, the iterates are given by 

$$m_{\phi_n}=(m_{\phi})^n=\left(S \Lambda S^{-1} \right)^n=S \Lambda^nS^{-1}.$$

Our goal then is to extend this definition by finding an expression for $\Lambda^n$ and replacing $n=0,1,2, \dots,$ with $t \in [0, \infty)$. The form of $\Lambda$ will depend on which of the seven cases we are in. Since the cases where we have Denjoy-Wolff point without multiplicity three can be written as direct sums of lower dimensional associated matrices, our primary concern will consist of defining a one parameter semigroup for cases where we have multiplicity three. Nonetheless, we will demonstrate our construction for all of the cases with Denjoy-Wolff point on the boundary.

In order to define a one parameter semigroup, we must be sure that our map stays in our space for fractional iterates and, more generally, for $\phi_t$ when $t \in [0, \infty)$  with $t \notin \mathbb{N} \cup \{0\}$. Recall that our characteristic domains consist of the whole space, the half space, and the Siegel half space. These are convex domains in $\mathbb{C}^2$. We make the following straightforward calculation.

Given two vectors $(u_1, u_2)$  and $(w_1,w_2)$ in the half space, we have for $t \in [0,1]$

$$\Re (tu_1+(1-t)w_1)=t\Re u_1+(1-t) \Re w_1 >0$$

\noindent and thus the half space is convex.

Given two vectors $(u_1, u_2)$  and $(w_1,w_2)$ in the Siegel half space, we have for $t \in [0,1]$

\begin{align*}
\Re (tu_1+(1-t)w_1)&=t\Re u_1+(1-t) \Re w_1 >t|u_2|^2+(1-t)|w_1|^2 \\
 &\geq t^2|u_2|^2+(1-t)^2|w_1|^2 \geq |t u_2 +(1-t)w_2|^2
\end{align*}

\noindent and thus the Siegel half space is convex. 

It is clear that, since in each case $\Lambda$ is in Jordan form (taken so that the off-diagonal elements are ones on the subdiagonal above the diagonal), each of the model maps $\Lambda$ are associated with a map of the form $Az+B$.

In cases where we have three distinct fixed points, $m_{\phi}$ is diagonalizable. Suppose 

$$\Lambda=\begin{pmatrix}
        \lambda_1 & 0 & 0 \\
        0 & \lambda_2 & 0 \\
        0 & 0 & \lambda_3
     \end{pmatrix}.$$

For $n$, a nonnegative integer, we have $m_{\phi_n}=(m_{\phi})^n=S\Lambda^n S^{-1}$ where

$$\Lambda^n=\begin{pmatrix}
        \lambda_1^n & 0 & 0 \\
        0 & \lambda_2^n & 0 \\
        0 & 0 & \lambda_3^n
     \end{pmatrix}.$$

We may embed this in a continuous semigroup defined by $m_{\phi_t}=S\Lambda^t S^{-1}$ for all $t \ge 0$ where 

$$\Lambda^t=\begin{pmatrix}
        \lambda_1^t & 0 & 0 \\
        0 & \lambda_2^t & 0 \\
        0 & 0 & \lambda_3^t
     \end{pmatrix}.$$

We then have for $s$, $t \ge 0$, 

$$m_{\phi_t}m_{\phi_s}=S\Lambda^t S^{-1}S\Lambda^sS^{-1}=S\Lambda^{t+s}S^{-1}=m_{\phi_{t+s}}$$

\noindent from which it follows that $\phi_{t \circ s}=\phi_{t+s}$. 

In the cases of multiplicity two, we note that $\Lambda$ is given by

$$\Lambda=\begin{pmatrix}
        \lambda & 1 & 0 \\
        0 & \lambda &  0 \\
        0 & 0 & \eta
     \end{pmatrix}$$

\noindent where $\lambda$ may equal $\eta$ depending on our characteristic domain. Hence we find

$$m_{\phi_n}=S\Lambda^n S^{-1}=S (A^n \oplus B^n) S^{-1}$$

\noindent with 

\begin{equation}
A^n=\begin{pmatrix}
        \lambda^n & n\lambda^{n-1}  \\
        0 & \lambda^n
     \end{pmatrix}
   \quad\mathrm{and}\quad 
B^n=(\eta^n).
\end{equation}

Hence we can embed this into a continuous semigroup defined for all $t \ge 0$ by

$$\Lambda=\begin{pmatrix}
        \lambda^t & t\lambda^{t-1} & 0 \\
        0 & \lambda^t &  0 \\
        0 & 0 & \eta^t
     \end{pmatrix}.$$

We check this satisfies the one parameter semigroup properties by noting

$$\Lambda^t \Lambda^s=\begin{pmatrix}
        \lambda^t & t\lambda^{t-1} & 0 \\
        0 & \lambda^t &  0 \\
        0 & 0 & \eta^t
     \end{pmatrix}
\begin{pmatrix}
        \lambda^s & s\lambda^{s-1} & 0 \\
        0 & \lambda^s &  0 \\
        0 & 0 & \eta^s
     \end{pmatrix}
=\begin{pmatrix}
        \lambda^{t+s} & (t+s)\lambda^{t+s-1} & 0\\
        0 & \lambda^{t+s} & 0 \\
        0 & 0 & \eta^{t+s}
     \end{pmatrix}=\Lambda^{t+s}$$

\noindent and thus

$$m_{\phi_t}m_{\phi_s}=S\Lambda^t S^{-1}S\Lambda^sS^{-1}=S\Lambda^{t+s}S^{-1}=m_{\phi_{t+s}}$$

\noindent from which it follows that $\phi_{t \circ s}=\phi_{t+s}$. 

In the cases where we have one fixed point of multiplicity three, $\Lambda$ has the form 

$$\Lambda=\begin{pmatrix}
        \alpha & 1& 0 \\
        0 & \alpha & 1 \\
        0 & 0 & \alpha
     \end{pmatrix}.$$

We may assume that the diagonal elements are all $1$ as, for $\lambda=\frac{1}{\alpha}$, we have

$$\Lambda=\begin{pmatrix}
        \alpha & 1& 0 \\
        0 & \alpha & 1 \\
        0 & 0 & \alpha
     \end{pmatrix}
=\alpha\begin{pmatrix}
        1 & \lambda & 0 \\
        0 & 1 & \lambda \\
        0 & 0 & 1
     \end{pmatrix}$$

\noindent and any multiple of an associated vector in $\mathbb{C}^{N+1}$ is associated with the same vector in $\mathbb{C}^N$. Thus, without loss of generality, 

$$\Lambda=\begin{pmatrix}
        1 & \lambda & 0 \\
        0 & 1 & \lambda \\
        0 & 0 & 1
     \end{pmatrix}.$$

We then recall a straightforward result that we reproduce for our specific case:

\begin{thm} For $n$, a nonnegative integer, we have $m_{\phi_n}=(m_{\phi})^n=S\Lambda^n S^{-1}$ where

$$\Lambda^n=\begin{pmatrix}
        1 & \lambda n & \frac{\lambda^2n(n-1)}{2} \\
        0 & 1 & \lambda n \\
        0 & 0 & 1
     \end{pmatrix}.$$

\end{thm}

\begin{proof}
We use proof by induction. Certainly our result is true for $n=1$. Now suppose that it is true for $n=k$. Then

\begin{align*}
\Lambda^{k+1}&=\Lambda^k\Lambda=\begin{pmatrix}
        1 & \lambda k & \frac{\lambda^2k(k-1)}{2} \\
        0 & 1 & \lambda k \\
        0 & 0 & 1
     \end{pmatrix}\begin{pmatrix}
        1 & \lambda & 0 \\
        0 & 1 & \lambda \\
        0 & 0 & 1
     \end{pmatrix}
=\begin{pmatrix}
        1 & \lambda+\lambda k & \lambda^2k+\frac{\lambda^2k(k-1)}{2} \\
        0 & 1 & \lambda+\lambda k \\
        0 & 0 & 1
     \end{pmatrix}\\
&=\begin{pmatrix}
        1 & \lambda(k+1) & \frac{\lambda^2}{2}\left(2k+k(k-1)\right) \\
        0 & 1 & \lambda(k+1)  \\
        0 & 0 & 1
     \end{pmatrix}=\begin{pmatrix}
        1 & \lambda(k+1) & \frac{\lambda^2(k+1)((k+1)-1)}{2} \\
        0 & 1 & \lambda(k+1)  \\
        0 & 0 & 1
     \end{pmatrix}
\end{align*}

and thus our result is true for $n=k+1$.

\end{proof}

We may embed this in a continuous semigroup defined by $m_{\phi_t}=S\Lambda^t S^{-1}$ for all $t \ge 0$ where 

$$\Lambda^t=\begin{pmatrix}
        1 & \lambda t & \frac{\lambda^2t(t-1)}{2} \\
        0 & 1 & \lambda t \\
        0 & 0 & 1
     \end{pmatrix}.$$

We check this satisfies the one parameter semigroup properties by noting

\begin{align*}
\Lambda^t \Lambda^s&=\begin{pmatrix}
        1 & \lambda t & \frac{\lambda^2t(t-1)}{2} \\
        0 & 1 & \lambda t \\
        0 & 0 & 1
     \end{pmatrix}
\begin{pmatrix}
        1 & \lambda s & \frac{\lambda^2s(s-1)}{2} \\
        0 & 1 & \lambda s \\
        0 & 0 & 1
     \end{pmatrix}\\
&=\begin{pmatrix}
        1 & \lambda(t+s) & \frac{\lambda^2(t+s)(t+s-1)}{2} \\
        0 & 1 & \lambda(t+s) \\
        0 & 0 & 1
     \end{pmatrix}=\Lambda^{t+s}
\end{align*}

\noindent and thus

$$m_{\phi_t}m_{\phi_s}=S\Lambda^t S^{-1}S\Lambda^sS^{-1}=S\Lambda^{t+s}S^{-1}=m_{\phi_{t+s}}$$

\noindent from which it follows that $\phi_{t \circ s}=\phi_{t+s}$.

\begin{exm}  \normalfont
Let $\phi$ be the linear fractional map from $\mathbb{B}_2$ into $\mathbb{B}_2$ given by

$$\phi(z)=\left(\frac{z_1+2z_2+1}{-z_1+2z_2+3}, \frac{-2z_1+2z_2+2}{-z_1+2z_2+3} \right).$$

Note that $\phi(1,0)=(1,0)$ with multiplicity three. We have 

$$\phi(z)=\frac{Az+B}{C^*z+D}$$

\noindent where $A=\begin{pmatrix}
        \phantom{-}1 & 2 \\
        -2 & 2 
     \end{pmatrix}$, $B=\begin{pmatrix}
        1 \\
        2 
     \end{pmatrix}$, $C=\begin{pmatrix}
        -1  \\
        \phantom{-}2 
     \end{pmatrix}$ and $D=3$.

Then, factoring our associated matrix into Jordan form, we obtain

$$m_{\phi}=\begin{pmatrix}
        \phantom{-}1 & 2 & 1 \\
        -2 & 2 & 2 \\
        -1 & 2 & 3
     \end{pmatrix}=\begin{pmatrix}
        1 & 0 & -\frac{1}{4} \\
        0 & \frac{1}{2} & -\frac{1}{8} \\
        1 & 0 & \phantom{-}0
     \end{pmatrix} \begin{pmatrix}
        2 & 1 & 0 \\
        0 & 2 & 1 \\
        0 & 0 & 2
     \end{pmatrix}\begin{pmatrix}
        \phantom{-}0 & 0 & 1 \\
        -1 & 2 & 1\\
        -4 & 0 & 4
     \end{pmatrix}$$

\noindent so after standardizing we obtain $\lambda=\frac{1}{2}$ and we have

$$m_{\phi_t}=\begin{pmatrix}
        1 & 0 & -\frac{1}{4} \\
        0 & \frac{1}{2} & -\frac{1}{8} \\
        1 & 0 & \phantom{-}0
     \end{pmatrix} \begin{pmatrix}
        1 & \frac{t}{2} & \frac{t(t-1)}{8} \\
        0 & 1 & \frac{t}{2} \\
        0 & 0 & 1
     \end{pmatrix}\begin{pmatrix}
        \phantom{-}0 & 0 & 1 \\
        -1 & 2 & 1\\
        -4 & 0 & 4
     \end{pmatrix}=\begin{pmatrix}
       \frac{2-t^2}{2} & t & \frac{t^2}{2} \\
        -t & 1 & t\\
        -\frac{t^2}{2} & t & \frac{t^2+2}{2}
     \end{pmatrix}$$

\noindent which implies

$$\phi_t(z_1, z_2)=\left( \frac{(2-t^2)z_1+2tz_2+t^2}{-t^2z_1+2tz_2+t^2+2}, \frac{-2tz_1+2z_2+2t}{-t^2z_1+2tz_2+t^2+2} \right).$$

It is a straightforward calculation to see that $\phi_0=I$ and $\phi_1=\phi$.

\end{exm}

\section{Constructing One Parameter Semigroups for Analytic Maps}

Since all analytic self maps of the disk have a model linear fractional map, it is reasonable to suppose that analytic self maps of the unit ball in higher dimensions {\it that have} a linear fractional model as in Theorem \ref{data} share common behavior with analytic self maps of the disk. The question of when a generic self map of the unit ball $\mathbb{B}_N$ has a linear fractional model is still open. We thus make the presumption that our analytic self map of $\mathbb{B}_2$ has such a model. For $\phi$ with a linear fractional model, it is still true that we may write $\sigma \circ \phi_n= \Phi_n \circ \sigma$. We deviate from the case of linear fractional maps, however, since we can no longer make the presumption that $\sigma$ is invertible in order to write $\phi=\sigma^{-1} \circ \Phi \circ \sigma$. We still can, for sufficiently large $n$, construct a one parameter semigroup for analytic $\phi$. We are guaranteed this by noting that, according to the model, $\Phi$ is a linear fractional automorphism of $\Omega$, implying $\Phi^{-1}(\Omega)=\Omega$. Now $\sigma$ maps $\mathbb{B}_2$ into $\Omega$ which means that $\Phi^{-1}(\sigma (\mathbb{B}_2)) \subset \Phi^{-1}(\Omega)=\Omega$ which implies $\Phi^{-n}(\sigma (\mathbb{B}_2)) \subset \Omega$ for $n$ a positive integer. By the conditions of Theorem \ref{data}, we conclude that $\Omega=\cup_{n=1}^{\infty}\Phi^{-n}(\sigma(\mathbb{B}_2))$. We note that for any point $z$ in $\Omega$, there is a sufficiently large $n$ such that $\Phi_n(z) \in \sigma(\mathbb{B}_2)$ from which we can construct an inverse. Thus, for sufficiently large $n$, we have 

$$\phi_n=\sigma^{-1} \circ \Phi_n \circ \sigma$$

\noindent which allows us to construct our one parameter semigroup in the same manner as in the case of linear fractional maps. Since we know that our model maps in $\mathbb{C}^2$ stay in our appropriate space for fractional iterates, for sufficiently large $n$, we can define a one parameter semigroup for analytic self map of $\mathbb{B}_2$ that have such an intertwining.

Below we construct a one parameter semigroup for an analytic map corresponding to the case with multiplicity three and characteristic domain $\mathbb{H}^2$ by explicitly constructing $\sigma$ and $\Phi$. As we will see in our construction, this analytic self map of $\mathbb{B}_2$ won't even be a rational map.

\begin{exm}  \normalfont
Recall that $\mathbb{B}_2$ is biholomorphic to the Siegel half space $\mathbb{H}^2$ via the Cayley map $\Psi(z)=\frac{z+1}{1-z_1}$. We then have $\Psi^{-1}(z)=\left(\frac{z_1-1}{z_1+1}, \frac{2z_2}{z_1+1} \right)$ and we let $\sigma=\omega \circ \Psi$ be our intertwining map. We proceed to construct our analytic map $\phi$ by taking the ``long route'' around the below commutative diagram:

$$\begin{tikzcd}
\mathbb{B}_2 \arrow[r, "\phi"] \arrow[d, "\sigma"]
& \mathbb{B}_2 \arrow[d, "\sigma"] \\
\mathbb{H}^2  \arrow[r, "\Phi"]
& \mathbb{H}^2
\end{tikzcd}$$

We will choose $\omega (z) = (\sqrt{2 z_1}, \sqrt{z_2})$ where we take the principal branch of the square root. One can show $\omega: \mathbb{H}^2 \rightarrow \mathbb{H}^2$. Then $\omega^{-1}(z)=(\frac{1}{2}z_1^2, z_2^2)$ with

\begin{equation}
\sigma(z)=\left( \sqrt{\frac{2(z_1+1)}{-z_1+1}}, \sqrt{\frac{z_2}{-z_1+1}} \right)
   \quad\mathrm{and}\quad 
\sigma^{-1}(z)=\left( \frac{z_1^2-2}{z_1^2+2}, \frac{4z_2^2}{z_1^2+2} \right).
\end{equation}

Next, we want to choose $\Phi$ to correspond to the case with multiplicity three and characteristic domain $\mathbb{H}^2$. By \cite{crosby}, we know that for this case, $\phi$ is equivalent to a Heisenberg translation whose associated matrix has one Jordan block. Recall that a Heisenberg translation in $\mathbb{C}^2$ is a linear fractional map of the form $h_b(z)=Az+b$ where $A=\begin{pmatrix}
        1 & 2\overline{b_2} \\
        0 & 1 
     \end{pmatrix}$ and $b=(b_1, b_2)^T$. Thus we choose our map $\Phi$ to be the Heisenberg translation given by
$$\Phi(z)=\left( z_1+\frac{1}{2}z_2+\frac{1}{2} , z_2+\frac{1}{4} \right).$$

We then define $\phi = \sigma^{-1} \circ \Phi \circ \sigma$. A calculation shows $\phi(z)=\left( \phi_1(z), \phi_2(z) \right)$ where

$$\phi_1(z)=\frac{15z_1+z_2+1+4\sqrt{2z_2(z_1+1)}+4\sqrt{2(1-z_1^2)}+2\sqrt{z_2(1-z_1)}}{-z_1+z_2+17+4\sqrt{2z_2(z_1+1)}+4\sqrt{2(1-z_1^2)}+2\sqrt{z_2(1-z_1)}} $$

\noindent and

$$\phi_2(z)=\frac{16z_2-z_1+1+8\sqrt{z_2(1-z_1)}}{-z_1+z_2+17+4\sqrt{2z_2(z_1+1)}+4\sqrt{2(1-z_1^2)}+2\sqrt{z_2(1-z_1)}}.$$

To define a one-parameter semigroup, we note that 

$$m_{\Phi}=\begin{pmatrix}
        1 & \frac{1}{2} & \frac{1}{2} \\
        0 & 1 & -\frac{1}{4} \\
        0 & 0 & 1
     \end{pmatrix}=\begin{pmatrix}
        1 & 0 & \phantom{-}0\\
        0 & 2 & -8 \\
        0 & 0 & \phantom{-}8
     \end{pmatrix} \begin{pmatrix}
        1 & 1 & 0 \\
        0 & 1 & 1 \\
        0 & 0 & 1
     \end{pmatrix}\begin{pmatrix}
        1 & 0 & 0 \\
        0 & \frac{1}{2} & \frac{1}{2} \\
        0& 0 & \frac{1}{8} 
     \end{pmatrix}$$

\noindent and thus

$$m_{\Phi_t}=\begin{pmatrix}
        1 & 0 & \phantom{-}0\\
        0 & 2 & -8 \\
        0 & 0 & \phantom{-}8
     \end{pmatrix} \begin{pmatrix}
        1 & t & \frac{t(t-1)}{2} \\
        0 & 1 & t \\
        0 & 0 & 1
     \end{pmatrix}\begin{pmatrix}
        1 & 0 & 0 \\
        0 & \frac{1}{2} & \frac{1}{2} \\
        0& 0 & \frac{1}{8} 
     \end{pmatrix}=\begin{pmatrix}
       1	& \frac{t}{2} & \frac{t(t+7)}{16} \\
        0 & 1 & \frac{t}{4}\\
        0 & 0 & 1
     \end{pmatrix}$$

\noindent which gives 

$$\Phi_t(z)=\left(z_1+\frac{t}{2}z_2+\frac{t(t+7)}{16}, z_2+\frac{t}{4} \right).$$

Hence we define $\phi_t=\sigma^{-1} \circ \Phi_t \circ \sigma$ to be our one-parameter semigroup of $\phi$.  Define the following:

\begin{align*}
A:=&1024z_1+64t^2z_2+t^2(t+7)^2(1-z_1)+256t\sqrt{2z_2(z_1+1)}\\
&+32t(t+7)\sqrt{2(1-z_1^2)}+16t^2(t+7)\sqrt{z_2(1-z_1)}\\
B:=&1024+64t^2z_2+t^2(t+7)^2(1-z_1)+256t\sqrt{2z_2(z_1+1)}\\
&+32t(t+7)\sqrt{2(1-z_1^2)}+16t^2(t+7)\sqrt{z_2(1-z_1)}\\
C:=&64t^2(1-z_1)+1024z_2+512t\sqrt{z_2(1-z_1)}\\
D:=&1024+64t^2z_2+t^2(t+7)^2(1-z_1)+256t\sqrt{2z_2(z_1+1)}\\
&+32t(t+7)\sqrt{2(1-z_1^2)}+16t^2(t+7)\sqrt{z_2(1-z_1)}.
\end{align*}

A calculation shows $\phi_t(z)=\left( \phi_{1_t}(z), \phi_{2_t}(z) \right)$ where $\phi_{1_t}(z)$ and $\phi_{2_t}(z)$ are given by

$$\phi_{1_t}(z)=\frac{A}{B}$$

\noindent and 

$$\phi_{2_t}(z)=\frac{C}{D}$$

It is a straightforward calculation to see that $\phi_0=I$ and $\phi_1=\phi$ and we have thus constructed a one parameter semigroup for an analytic map of $\mathbb{B}_2$ into itself with a corresponding model linear fractional map.

\end{exm}

\section{Further Questions}

While we have realized an explicit construction of one parameter semigroups for linear fractional maps in $\mathbb{B}_2$ and for analytic self maps of $\mathbb{B}_2$ that have a linear fractional model, much remains to be done. The obstruction to generalizing these results to $\mathbb{C}_N$ lies in the fact that the classification of linear fractional maps in $\mathbb{B}_2$ according to \cite{crosby} is not yet generalized to $\mathbb{B}_N$. One suspects, however, that these results should generalize in an appropriate manner.

\begin{quest}
All linear fractional self maps of $\mathbb{B}_N$ for $N=1,2$ have a linear fractional model. Can this be generalized to $\mathbb{B}_N$ for all $N \in \mathbb{N}$?
\end{quest}

If this question can be answered in the affirmative, then the technique of the paper would allow an explicit construction of one parameter semigroups for all linear fractional self maps, not constant and not an automorphism, of the unit ball in any dimension.

For analytic self maps of the disk, we are guaranteed a linear fractional model. This is also true for linear fractional maps of $\mathbb{B}_2$ into itself. Under what conditions does this generalize?

\begin{quest}
What conditions must be imposed on an analytic self map $\phi$  of $\mathbb{B}_N$ for $\phi$ to be guaranteed a linear fractional model?
\end{quest}

As an exercise to the reader, it can be shown that for an $m \times m$ Jordan block $A$ and for $n \in \mathbb{N}$, the set $\{A^n\}$ can be extended to a one parameter semigroup. This allows the techniques of this paper to generalize for the case in which an analytic self map of $\mathbb{B}_N$ has such a linear fractional model.

\end{document}